\documentclass[11pt]{amsart}
\usepackage{setspace}
\usepackage{amsmath,amsfonts,amscd,amssymb,bm,amsthm}
\usepackage{tikz-cd}
\usepackage{enumitem}

\usepackage[all,color]{xy}
\UseCrayolaColors
\usetikzlibrary{positioning,arrows,calc}
\usepackage{graphicx}
\usepackage{comment}
\usepackage[pdfencoding=auto,colorlinks,
            linkcolor=red,
            anchorcolor=blue,
            citecolor=green]{hyperref}
\usepackage{bookmark}

\usepackage[backend=biber]{biblatex}\bibliography{Pi1.bib}
\usepackage[all]{xy}
\usepackage{tikz}
\definecolor{myred}{RGB}{183,18,52}
\definecolor{myyellow}{RGB}{254,213,1}
\definecolor{myblue}{RGB}{0,80,198}
\definecolor{mygreen}{RGB}{0,155,72}
\newcommand{\mS}{\mathcal{S}}

\newcommand{\cod}{\mbox{\rm cod}}
\newcommand{\Si}{{\Sigma}}
\newcommand{\Gr}{{\rm Gr}}

\newcommand{\io}{{\iota}}
\newcommand{\Ff}{{\mathcal F}}

\newcommand{\Jj}{{\mathcal J}}
\newcommand{\eps}{{\epsilon}}
\newcommand{\mT}{\mathcal{T}}
\newcommand{\mU}{\mathcal{U}}
\newcommand{\mC}{\mathcal{C}}
\newcommand{\Diff}{{\rm Diff}}
\newcommand{\Symp}{{\rm Symp}}

\newcommand{\mA}{\mathcal{A}}

\newcommand{\mF}{\mathcal{F}}
\newcommand{\mK}{\mathcal{K}}
\newcommand{\mE}{\mathcal{E}}

\newcommand{\mD}{\mathcal{D}}

\newcommand{\CC}{\mathbb{C}}

\newcommand{\RR}{\mathbb{R}}

\newcommand{\Dd}{{\mathcal D}}

\newcommand{\w}{\omega}
\newcommand{\om}{\omega}

\newcommand{\Fol}{{\rm Lef}}

\newtheorem{thm}{Theorem}[section]
\newtheorem{dfn}[thm]{Definition}
\newtheorem{cor}[thm]{Corollary}
\newtheorem{lma}[thm]{Lemma}
\newtheorem{Claim}[thm]{Claim}
\newtheorem{prp}[thm]{Proposition}

\newtheorem{rmk}[thm]{Remark}
\newtheorem{conj}[thm]{Conjecture}

\begin{document}
\title[Symplectic isotopy on ruled surfaces]{Symplectic isotopy on non-minimal ruled surfaces  }
\author{ Olguta Buse and Jun Li  }

\date{\today}
\begin{abstract}
  We prove the stability of $\Symp(X,w)\cap \Diff_0(X)$ for a one-point blow-up of irrational ruled surfaces and study their topological colimit. Non-trivial generators of $\pi_0[\Symp(X,w)\cap \Diff_0(X)]$ that differ from Lagrangian Dehn twists are detected.
 \end{abstract}

\maketitle
\tableofcontents

\section{Introduction}\label{Intro}

In this note, we study some topological aspects of symplectomorphism groups of non-minimal irrational ruled surfaces.

Let $M_g=\Sigma_g \times S^2 $ be a topologically trivial  ruled surface and $M_g\# \overline{n \CC P^2}$ its $n$-point blow-up, with homology classes $B, F, E_1 \ldots E_n$ for the section, fiber and exceptional curves. By McDuff's classification results for minimal ruled surfaces, \cite{McD94}, any symplectic form is diffeomorphic to $\mu \sigma_{\Sigma_g} \oplus \sigma_{S^2}$ for some $ \mu >0$, up to diffeomorphisms and normalization. Following work of Li- Liu \cite{LLiu2}, such classification result also holds in the  blow-ups $M_g\# \overline{n \CC P^2}$. Namely, up to  diffeomorphism and normalization, a symplectic form on the blow-up is represented by a vector  $u=[\mu, 1, c_1, \cdots, c_n]$ \footnote{The components of $[\mu, 1, c_1, \cdots, c_n)] $ are the areas of the homology class $B, F, E_1, \cdots, E_n$, where $\mu>0, c_1+c_2<1, 0<c_i<1,$ $ c_1\geq c_2\geq \cdots \geq c_n, 2 \mu> c_1^2+ \ldots c_n^2 $} belonging  to a convex region 
in $\RR^{n+1}$ which describes the normalized reduced symplectic cone.

  Following the results of \cite{Abr98, AM99, McDacs, LP04, ALP, AGK09, Buse11, AP13, ALLP} etc, the present note addresses the topological behavior of the symplectomorphism groups as the form $\omega _u$  varies within the symplectic cone.

  {\em Informally speaking, we conjecture that for all blow-up ruled surfaces there exists a partition of the symplectic cone into chambers such that the homotopy type of the symplectomorphism groups does not change within a chamber.  Hence there exists a ``topological colimit'' when we stretch the area of $\Sigma_g$ to be very large. This further allows us to probe the symplectic isotopy question on non-minimal ruled surfaces by exhibiting smoothly but not symplectically trivial mapping classes that survive in the colimit.  }

  \subsection{Statements of the problems}

  We partition the reduced normalized symplectic cone, which we will name  $\Delta^{n+1}$, into countably many half open, half closed chambers by linear equations in $\RR^{n+1}$ such that each chamber has the same arithmetically admissible cohomology classes, i.e. positive when evaluated on the corresponding homology classes. 

\begin{figure}[ht]
  \centering
\[
\xy
(-20, -20)*{};(100, -20)* {}**\dir{--};
  (109, -20)* {\mu\to \infty};
(0, 0)*{}; (0, -20)*{}**\dir{-};
(0, -9)*{B-F};
(-20, -20)*{};(-10, 0)** {}**\crv{~*=<4pt>{.} (-18,-5)};
(0, -20)*{};(20, 0)** {}**\dir{-};
(20, -20)*{};(40, 0)** {}**\dir{-};
(40, -20)*{};(60, 0)** {}**\dir{-};
(20, 0)*{}; (20, -20)*{}**\dir{-};
(23, -9)*{B-2F};
(40, 0)*{}; (40, -20)*{}**\dir{-};
(43, -9)*{B-3F};
(-10, 0)*{};(100, 0)* {}**\dir{--};
(60, 0)*{}; (60, -20)*{}**\dir{-};
(60, -25)*{\cdots \cdots};
(80, 0)*{}; (80, -20)*{}**\dir[red, ultra thick, domain=0:6]{-};
(80, 3)*{\cdots \cdots };
  (-20, -23)*{\mu=0};
  (10, -14)*{\rotatebox{45}{\text{B-F-E}}};
  (30, -14)*{\rotatebox{45}{\text{B-2F-E}}};
\endxy
\]
  \caption{(Normalized) Symplectic cone of one-point blow-up}
  \label{cfigd}
\end{figure}

Since most results of the paper concern the one-point blow-up, we begin by introducing Figure \ref{cfigd}, which shows the significant region for the normalized symplectic cone corresponding to the one-point blow-up. The region represents (normalized) cohomology classes in  $\RR^{2}$ and the walls are labeled with homology classes on which the symplectic form evaluates trivially.   See Section \ref{conepart} for details.

 The main problem addressed in this paper is that of topological stability of the symplectomorphism group  as the cohomology class  $u=[\mu,1, c_1, \cdots c_n]$ of the symplectic form varies within the chambers.   Since Li-Liu's results ensure a diffeomorphism between any cohomologous symplectic forms, the corresponding symplectomorphism groups are homeomorphic. Hence for any given class  $u=[\mu,1, c_1, \cdots c_n]$  we call $G_{u,g}^n$  be the intersection of $\Symp(M_g\# \overline{n \CC P^2},\w)$ with $\Diff_0(M_g\# \overline{n \CC P^2})$, the identity component of the diffeomorphism group.  We will tackle the following conjecture on the topology of  $G_{u,g}^n$:

\begin{conj}\label{stabconj}
 Let M be either $M_g $ or one of its blow-ups,   If two symplectic forms on $M$ are represented by $u_1$ and $u_2$ belonging to the same chamber then $\pi_i(G_{u_1,g}^n)= \pi_i(G_{u_2,g}^n)$, for all $i\geq 0$.
\end{conj}

Some variations of this conjecture can be addressed. For instance, one can prove stability for only a selected collection of these regions or only address the first $n$ levels of homotopy groups.

This conjecture for minimal models has been established by McDuff \cite{McDacs}  for $g=0$ and by Buse \cite{Buse11} for $g>0$. The conjecture for the rational surface cases with Euler number up to 12 has been proved by Anjos-Li-Li-Pinsonnault in   \cite{ALLP}.

 In the present paper, we establish this conjecture for the one-point blow-up cases. The many-point blow-up cases will be studied in a future work \cite{BL2}.

\begin{thm}\label{highgenus}
  Conjecture \ref{stabconj}  holds for  $M_g\# \overline{ \CC P^2}$, for all $g \geq 1$, with a symplectic form $\w$ such that $[\w]=[\mu, 1, c]$ is in the region $\mu >g$, i.e. $\w(\Si_g)> g\w(S^2).$

  More concretely, except for the first $2g-1$ chambers of Figure \ref{cfigd}, the stability holds in the two types of chambers shown in  Figure \ref{cfige} below.

Further, for all $\w$ with $u=[\w]=[\mu, 1, c]$ and $\w'$ with  $u'=[\w']=[\mu+\epsilon, 1, c], \epsilon>0$ and for all $i \le 2\lfloor \mu \rfloor -2g,$ we have $\pi_i(G_{u,g}^1)= \pi_i(G_{u',g}^1)$.

\begin{figure}[ht]
  \centering
\[
\xy
(20, -20)*{};(40, -20)* {}**\dir{--};
(0, 0)*{}; (0, -20)*{}**\dir{--};
(0, -20)*{};(20, 0)** {}**\dir{-};
 (-4, -12)*{{\text{B-pF}}};
  (10, -14)*{\rotatebox{45}{\text{B-pF-E}}};
  (30, -14)*{\rotatebox{45}{\text{B-qF-E}}};
   (48, -12)*{{\text{B-(q+1)F}}};
(20, -20)*{};(40, 0)** {}**\dir{--};
(40, 0)*{}; (40, -20)*{}**\dir{-};
(0, 0)*{};(20, 0)* {}**\dir{--};
\endxy
\]
  \caption{ Stability chambers for  $M_g\# \overline{ \CC P^2}$, $p,q \ge g$}
  \label{cfige}
\end{figure}

\end{thm}

\newpage
Once we establish the stability Theorem \ref{highgenus}, this allows us to show the following topological colimit characterization.

\begin{thm}\label{tlimitintro}
  For $M_g\#\overline{\CC P^2}$, for all $g\geq 1$, there exists a topological colimit $G^1_{\infty,g}$ for the groups $G^1_{u,g}$  as $\mu\to \infty$.  Furthermore, $G^1_{\infty,g}$ has the homotopy type of a diffeomorphism group that we will call $\mD^1_g$.
  Moreover $\mD^1_g$ is disconnected for $g>1.$
\end{thm}


In Section \ref{s:out}, we are going to give the details of the construction of $\mD^1_g$.  A sufficient understanding of the topological colimit  allows us to conclude the following:

\begin{cor}\label{discon}
  $G^1_{u,g}$ is disconnected for $g>1$ and when the class of the symplectic  $[\w]=[\mu, 1, c]$ is in the region $\mu >g$.
\end{cor}

It remains an open question whether the topological colimit for $g=1$ is connected or not.

In \cite{MS17} the following conjecture is proposed as open problem 14 for minimal ruled surfaces.

\begin{conj}[Symplectic isotopy conjecture for ruled surfaces]\label{minisoconj}
 For $(M_g,\w) $,   a symplectomorphism is symplectically isotopic to identity if and only if it is smoothly isotopic to identity.
\end{conj}

Here we remind the reader that the terminology "symplectic isotopy" here and in \cite{MS17} for the above conjecture refer to isotopy of symplectic maps, rather than of symplectic submanifold.

Our Corollary \ref{discon} shows that the corresponding conjecture is not true for the one-point blow-up of ruled surfaces.  There exist ``exotic symplectomorphisms'' that are smoothly but not symplectically isotopic to identity.  Notice that for topological reasons, there are no Lagrangian spheres inside $M_g\#\overline{\CC P^2}$, and hence those ``exotic symplectomorphisms'' are not generated by Dehn twists along Lagrangian spheres.

It is interesting to compare this result with Shevshishin-Simirov \cite{SS17elliptic}, who detected  exotic symplectomorphisms only in the case $g=1$ for $M_1\#\overline{\CC P^2}$. Their construction is quite different from ours and does not survive in the colimit.  It is also worth pointing out that our construction has a connection to the fibered Dehn twists of Biran-Giroux (cf. \cite{BirGir}).  See Remark \ref{fiberdehn} for more discussions.

\vspace{.3cm}

\subsection{Techniques of proofs}
The main difficulty in approaching the main conjecture is the absence of direct maps between symplectomorphism groups that correspond to two deformation equivalent symplectic forms on a given manifold.

McDuff's approach in  \cite{McDacs} was to consider the Kronheimer's fibration

\begin{equation}\label{fib}
  \Symp(M, \w)\cap \Diff_0(M) \to \Diff_0(M) \to \mT_{\w},
\end{equation}
where $ \mT_{\w}$ represents the space of symplectic forms in the class $[\w]$ and isotopic to a given form, and $\Diff_0(M)$ is the identity component of the diffeomorphism group (See  \cite{Kro99}). Moser's technique grants a transitive action of $\Diff_0(M)$ on  $ \mT_{\w}$ and hence gives us the fibration \eqref{fib}.

Following  McDuff's work in \cite{McDacs}\footnote{ McDuff's original results were written in terms of homotopy fibrations where the larger space of taming almost complex structures was used. Since this space is homotopy equivalent to the space of compatible structures we use (for reasons explained in Section \ref{mininf}) we will use the compatible structure spaces instead.}, one uses the (weak) homotopy equivalence between $ \mT_{\w}
$ and the space $\mA_{\w}$ (which is the space of $\w'$-compatible almost complex structures, where $\w'$ is any symplectic form isotopic to $\w$) to construct a homotopy fibration
\begin{equation}\label{fibacs}
    \Symp(M, \w)\cap \Diff_0(M) \to\Diff_0(M) \to \mA_{\w}.
\end{equation}

By the inflation technique of Section \ref{stabproof} one can relate (by direct inclusions of strata) the spaces $\mA_{\w}$'s for $\w$ in different cohomology classes and hence prove stability results on $\Symp(M,\w)$.
Thus the heart of the matter remains to establish such strata and inclusions in the given setting. That involves:

(a) establishing the existence of sufficient inflation techniques

(b) existence of sufficiently many $J$-holomorphic embedded (or nodal) curves for non-generic almost complex structures so that the inflation can be performed.

{\em Most techniques used here concern the study of spaces of almost complex structures (not necessarily generic ones) and the $J$-holomorphic curves they admit.}

\noindent\textbf{Acknowledgements.} The second named author is supported by an AMS-Simons travel grant. We are grateful to  Richard Hind, Tian-Jun Li, Dusa McDuff, Weiwei Wu, Weiyi Zhang for helpful conversations.  We thank an anonymous referee for his/her very careful reading and many helpful suggestions throughout the paper.

\section{The symplectic cone and its partition}\label{cone}

\subsection{The normalized reduced symplectic cone}

Let $K_0$ be the standard symplectic canonical class, $\mE$ be the set of all $K_0$ -symplectic exceptional $-1$ sphere classes, and $P$ be the set of positive two-forms, i.e. $P:= \{\alpha\in H^2(M, \RR), | \alpha \wedge \alpha >0\}.$  Let $\mC_M$ denote the symplectic cone.

In Theorem 4 of \cite{LL01}, Li-Liu showed that if $M$ is a closed, oriented 4-manifold with $b^+= 1$
and if the symplectic cone  $\mC_M$ is nonempty, then
$$\mC_M= \{e \in P |0 < |e \cdot E| \quad \text{for all} \quad E \in \mE \}.$$

Moreover, from work in \cite{LL01} the following is true:

\begin{prp}\label{coner}
 Choose the standard basis $B, F, E_1, \cdots, E_n$ in $M_g^n$ and associate area coefficients for a symplectic class $[\mu, \alpha, c_1, \cdots, c_n]$.

In  any such ruled symplectic manifold a symplectic form is always diffeomorphic to a reduced normalized  symplectic form whose class is obtained from:

\begin{itemize}

\item {\bf All  symplectic classes} are  given by $[\omega]^2>0$ and the  symplectic form is positive on all $K_0$ symplectic  $-1$ sphere classes given by $E_i$ and $F-E_i$,

\item {\bf All   reduced symplectic classes} are given by the extra condition that the  symplectic form is positive on all $K_0$ symplectic  $-2$ sphere classes given by $E_i-E_j$, $i<j$ and $F-E_i-E_j$.

\item We can moreover normalize the symplectic form such that $F$ has unitary symplectic area.
\end{itemize}

Summarizing the conditions, up to normalization and diffeomorphisms  all the vectors $[\mu, 1, c_1, \cdots, c_n] $ in  the normalized reduced cone
satisfy the following equations:
 $ c_1+c_2<1, 0<c_i<1, c_1\geq c_2\geq \cdots \geq c_n,$ and $  2 \mu> c_1^2+ \ldots c_n^2 .$

\end{prp}

The normalized reduced  vectors  $u=[\mu, 1, c_1, \cdots, c_n]$ belong to a convex region $\Delta^{n+1}$ in $\RR^{n+1}$, whose boundary walls are $n$-dimensional convex regions given by linear equations.

We will refer to $\Delta^{n+1}$ as the normalized reduced symplectic cone.
Its boundary walls are $n$-dimensional convex regions given by linear equations. We will be concerned with symplectic deformations inside this region $\Delta^{n+1}$ for the $n$-point blow-ups.

Note that we partition the normalized symplectic cone by looking at the homology classes of potential symplectic surfaces. To that end, we will now fix some notation:

\begin{dfn}\label{sw}

Let $\mathcal S_{\omega}$ denote the set of homology classes of embedded $\omega$-symplectic surfaces and $K_{\w}$ the symplectic canonical class. For any $A\in \mathcal S_{\omega}$, the adjunction formula gives
  \begin{equation}\label{AF}
    K_{\omega}\cdot A=-A\cdot A -2+2g(A).
\end{equation}

For each $A\in  \mathcal S_{\omega} $ we associate the integer
$$ {\cod_A}=2(-A\cdot A-1+g).$$

\end{dfn}
 By \cite{LL01}, if $M$ is a closed, oriented 4-manifold with $b^+= 1$, the symplectic canonical class is unique once we fix $u=[\w]$. We denote it by $K_u$.

    We denote $\mS^{<0}_\omega$ the subset of $\mS_\omega$ having negative self-intersections.

\begin{rmk} The sets introduced in  Definition \ref{sw} depend only on the cohomology class of $\omega.$
\begin{itemize}
  \item By Lemma 2.4 in \cite{ALLP} $\mS_w$ depends only on the cohomology class of $\omega.$
  \item To explain that $\mS^{<0}_\omega$ depends on the cohomological data of the symplectic class only, we follow the main result from Li-Usher \cite{LU06}. By performing standard blow ups from minimal surfaces we find some symplectic form $\omega'$ with an integral class  $[\omega']$ that admits some embedded symplectic curve $D$ in any negative self intersection class of  $\mS^{<0}_{\omega'}$. The symplectic inflation of
   \cite{LU06} allows us to produce symplectic forms in any other class   $u$ which pairs positively with $D$.


 \end{itemize}
 Consequently, we can define $\mS_u$, where $u=[\w]$, using only the cohomology data of $\w$.   We denote $\mS^{<0}_u$ the corresponding subset.
\end{rmk}

\begin{dfn} \label{wall}
We call an {\bf interior wall} a nonempty convex subset of the symplectic cone $\Delta^{n+1}$ given by $\{u| \ u\cdot A=0,$  when  $A$  belongs to some nonempty subset of  $\mS^{< 0}_u \}$.

 We call a {\bf boundary  wall} a
 smooth connected maximal dimension subset  of the complement of the interior of $ \Delta^{n+1}$ with respect to the  closure of $ \overline{\Delta^{n+1}}.$

 We will call $$ {\cod_u}= max \{\cod_A,  A\in \mS^{< 0}_u, u\cdot A \geq 0 \}.$$

 In particular, note that the $cod_u$  potential changes are precisely across  interior walls.
\end{dfn}

We partition the normalized reduced symplectic cone  by interior walls and boundary walls into chambers. This partition is more complicated for several points blow-ups.
For a one-point blowup, we now give an explicit description of the chambers.

 \subsection{Partition of the normalized reduced symplectic cone}\label{conepart}

When $M=M_g\# \overline{ \CC P^2},$   Figure \ref{cfigd1} represents $ \Delta^{2}.$  More precisely, following Proposition \ref{coner} it is given by vectors $[\mu,1,c]$ with $1>c>0, 2\mu>c^2$. Note that if $\mu >\frac1 2$ the second volume condition is satisfied automatically as $2 \mu \geq 1 >c^2.$
The following Figure \ref{cfigd1} gives  the cone in the one point blow up case:

\begin{figure}[ht]
  \centering
\[
\xy
(-20, -20)*{};(100, -20)* {}**\dir{--};
  (109, -20)* {\mu\to \infty};
(0, 0)*{}; (0, -20)*{}**\dir{-};
(0, -9)*{B-F};
(-20, -20)*{};(-10, 0)** {}**\crv{~*=<4pt>{.} (-18,-5)};
(0, -20)*{};(20, 0)** {}**\dir{-};
(20, -20)*{};(40, 0)** {}**\dir{-};
(40, -20)*{};(60, 0)** {}**\dir{-};
(20, 0)*{}; (20, -20)*{}**\dir{-};
(23, -9)*{B-2F};
(40, 0)*{}; (40, -20)*{}**\dir{-};
(43, -9)*{B-3F};
(-10, 0)*{};(100, 0)* {}**\dir{--};
(60, 0)*{}; (60, -20)*{}**\dir{-};
(60, -25)*{\cdots \cdots};
(80, 0)*{}; (80, -20)*{}**\dir[red, ultra thick, domain=0:6]{-};
(80, 3)*{\cdots \cdots };
  (-20, -23)*{\mu=0};
  (10, -14)*{\rotatebox{45}{\text{B-F-E}}};
  (30, -14)*{\rotatebox{45}{\text{B-2F-E}}};
\endxy
\]
  \caption{(Normalized) reduced symplectic cone of a one-point blow-up}\label{cfigd1}
\end{figure}

The dotted boundary curve on the left is provided by the volume constraints. The chambers that we consider  in the one-point blow-up cases all have $\mu >1$ and are half open half closed (to the  right) two-dimensional regions given explicitly;  the following inequalities are sufficient to describe these chambers:
\begin{itemize}
\item The top open boundary wall is given by $u\cdot E=1;$ (excluded from chambers ) while  the bottom open boundary wall is given by
 $u\cdot E=0$ ( also excluded from chambers.)

  \item The interior right wall for the $(2k+1)$-th chamber is provided by the equalities
    $u\cdot [B-(k+1)F-E]=0$ (included from the chamber ) while the left wall is  $u \cdot [B-(k+1)F] =0$ (excluded in the chamber).
  \item The  interior  walls for the $(2k)$-th are provided by the inequalities:
    $u\cdot [B-(k+1)F]=0$,(on the right, included) $u\cdot [B-kF-E] =0$(on the left, excluded).
\item We will say that $u$ is in the positive range of $B-kF$ or $B-kF-E$ respectively if $u \cdot (B-kF)>0$ or $u \cdot (B-kF-E )>0$ respectively. We will refer to such classes $u$ as being {\it in the $B-kF$ or $B-kF-E$ positive cones}.

\end{itemize}

Note that in Figure \ref{cfigd1}, the bottom dashed line consists of  degenerate cohomology classes (trivial on the class $E$) and the line mimics a symplectic minimal ruled surface which is the product $\Sigma_g\times S^2$ with $\w(\Sigma_g)/\w(S^2) =\mu.$  In turn, the top dashed line consists of homology classes trivial on $F- E$  and mimics a symplectic minimal ruled surface which is the non-trivial bundle $\Sigma_g \tilde{\times} S^2$ with $\w(\Sigma_g)/\w(S^2) =\mu.$    The interior of the cone is represented by symplectic forms on a one-point blow-up of the minimal ruled surface, such that $\w(\Sigma_g)=\mU$, $\w(S^2)=1,$ and $0<\w(E)=c<1.$

\begin{rmk}\label{singlecurve}

For one point blow up  these chambers can be labeled by a single  negative self intersection class since any two homologically distinct  negative self intersection curves have a negative intersection pairing.

\end{rmk}

\section{Homotopy fibration and the stratification of  $\mA_u$}

  The homotopy fibration \eqref{fibacs} presented in the introduction has been fruitfully applied to study the topology of $G_{\w}$, especially in dimension 4, cf. \cite{McDacs}, \cite{Buse11}, \cite{AP13}. Now let us focus on the case of non-minimal ruled surfaces. Following the same strategies employed by \cite{LL16}, one needs to find a sufficiently fine stratification of the spaces of almost complex structures and show that they only differ by the addition of said (finite codimension) strata when $\w$ crosses the walls of the chambers of the arithmetic regions in  $\Delta^{n+1}$.

We will consider the larger space $\mA_u$ of almost complex structures compatible with some $\w$ that satisfies $[\w]= u$. This space is potentially disconnected, and in a discussion on this topic included in Section \ref{b+1inf} we will show that we are able to identify a canonical path component within each $\mA_u.$

\subsection{ Stratification of the spaces of almost complex structures}

First we define the prime subset of the space of almost complex structures $\mA_{u, \mC}$ labeled by set $\mC\subset \mathcal S^{< 0}_u$ for a given isotopy class of $\w$ as follows:

 \begin{dfn} \label{fine decomposition}
  A subset  $\mC\subset \mathcal S^{< 0}_u$ is admissible if
  \[\mC=\{A_1, \cdots, A_i,\cdots ,A_q |\quad  A_i\cdot A_j \geq 0, \text{ for all }  i\neq j\}.\]
  Given an admissible subset $\mC$, we  define the real codimension
  of the label set ${\mC}$ as
  \[\cod({\mC})= \sum_{A_i\in \mC} \cod_{A_i}=\sum_{A_i\in \mC}  2(-A_i\cdot A_i-1+g_i).\]
  Define the {\bf prime subset} $\mA_{u, \mC}$ as the collection of $J \in \mA_{u}$ for which $A \in \mathcal S_{u}^{<0}$ has an embedded $J$-holomorphic representative if and only if $A\in \mC\}$.
If $\mC=\{A\}$ contains only one class $A$, we abbreviate $\mA_A$ for $\mA_{\{A\}}$.
 \end{dfn}

  Notice that these prime subsets are disjoint and
 we have the decomposition $\mA_{u} =\amalg_{\mC} \mA_{u,\mC}.$

We define the  open stratum  $\mA_{u, {\rm open}}$ as the complement of positive codimension strata in $\mA_{u}$.

 \begin{rmk}
  Note   that we do not have control of what type of classes with non-negative self intersection are $J$-holomorphic embedded in $\mA_{u, {\rm open}}$.  For different $J \in  \mA_{u, {\rm open}},$ there could be different embedded curves in the section classes $B, B+F, B+2F, \cdots;$  see proof of Proposition \ref{curvesexist} for details.
 \end{rmk}

 For $\mS^{< 0}_u$,  the following shows that the prime subsets are well-behaved analytically.
 \begin{prp}\label{stratum}

Let $(X,\omega)$ be a  ruled 4-manifold  such that $[\omega]=u$.
 Suppose $U_{\mC}\subset\mA_u$ is a subset
characterized by the existence of a configuration of embedded
$J$-holomorphic  curves $C_{1}\cup C_{2}\cup\cdots\cup C_{N}$ of  positive codimension as in Definition \ref{sw} with $\{ [C_{1}], [C_{2}],\cdots , [C_{N}]\}=\mC$.
Then $U_{\mC}$ is a co-oriented Fr\'echet suborbifold
   of $\mA_u$ of (real) codimension $2N-2c_{1}([C_{1}]+\cdots+ [C_{N}])=\sum_i K\cdot [C_i]-[C_i]^2.$

\end{prp}

\begin{proof}
  First, it suffices to show the above result in the Banach setting, see the convention in the introduction of \cite{AGK09}.  In particular, Theorem 2.1.2 in \cite{IS99} shows that the space of the nodal curve in those fixed classes is a finite codimensional Banach analytic subset.  Using the argument in Appendix B of \cite{AP13},
   at each point $J_0$ one can construct a local chart of $\mA_{u}$ with codimension
   $2N-2c_{1}([C_{1}]+\cdots+ [C_{N}])$, so that any almost complex structure in this chart, has an embedded pseudo-holomorphic curve in each homology class $[C_i]$.

The orbifold structure comes from the Teichmüller space
of the moduli space of Riemann surfaces of genus $g$. Hence the chart is a quotient by at most a finite group.   Note that this follows Lemma 2.6 in \cite{McDacs}.
\end{proof}

\begin{cor}
  The prime subsets $\mA_{u, \mC}$ are suborbifolds of $\mA_u$ having codimension equal to $\cod(\mC)$.
\end{cor}

\begin{proof}
$\mA_{u, \mC}$ is a subset of $U_{\mC},$ and the complement of
$\mA_{u, \mC}$ in  $U_{\mC}$  is a union of $U_{\mC_i},$ where the $\mC_i$ are admissible sets containing $\mC$.  Hence $\mA_{u, \mC}$ is a suborbifold of the same codimension as $U_{\mC}$ in $\mA_u.$
\end{proof}

For one-point blowup of ruled surface, as remarked in Section \ref{singlecurve}, each  $\mA_{u, \mC}$ is characterized by the existence of a curve in class $B-kF$ or $B-kF-E, k\ge 0.$  We point out here that for a given $\omega$, all its strata $\mA_{u, \mC}$ are non-empty and mutually disjoint, because we can find complex structures by blowing corresponding minimal ruled surfaces with section classes of self-intersections $2k$ or $2k+1.$

\subsection{Inflation on spaces of almost complex structures}\label{tcinf}

McDuff \cite{McDacs} and later Buse \cite{Buse11} used a version of symplectic inflation keeping track of an almost complex structure $J$ to study the spaces of almost complex structures.
As discussed in \cite{ALLP} Section 4.4,  their combined results amount to the following version of the inflation theorem:

  \begin{thm}\label{ctinf}  For a 4-manifold $M$, given a compatible pair  $(J,\w),$ one can inflate along a $J$-holomorphic curve $Z$, so that there exists a symplectic form $\w'$ taming  $J$ such that $[\w']= [\w]+ t PD(Z), t\in [0,\mu)$ where $\mu= \infty$ if $Z\cdot Z\ge0$ and $\mu= \frac{\w(Z)}{(-Z\cdot Z)}$ if $Z\cdot Z<0$.
\end{thm}

We remark that the bound on the cohomology class of $\omega'$ in the Theorem  \ref{ctinf} is automatically satisfied whenever $\omega'(Z)>0.$
 As further explained in \cite{ALLP}, considering the Li-Zhang Cone Theorem allows us to strengthen the above theorem. More concretely, recall the $J$-tame cone and the $J$-compatible cone:

\[
\mK_J^{t}=\{[\omega]\in H^2(M;\RR)|\hbox{$\omega$  tames
$J$}\},\]
\[\mK_J^{c}=\{[\omega]\in H^2(M;\RR)|\hbox{$\omega$  is compatible with $J$}\}.
\]
  Both $\mK_J^c$ and $\mK_J^t$ are convex cohomology cones, and clearly $\mathcal K_J^c\subset \mathcal K_J^t$.  The following theorem turns this into equality in the case of an almost K\"ahler $J$ on $M^4$ with $b^+=1$.
\begin{thm}[Theorem 1.3 in \cite{LZ09}]\label{ct}
Let $(M,J)$ be an almost complex 4-manifold.  If $b^+(M)=1$ and $\mK_J^c\ne \emptyset$, then $\mK_J^c=\mK_J^t$.
\end{thm}

For 4-dimensional manifolds $M$ with $b^+=1,$ this further allows us to enhance  the above theorem \ref{ctinf} to the {\bf $b^+=1$ $J$-compatible inflation theorem \ref{ccinf}}.  This in particular works for any blow-up of ruled surfaces.

  \begin{thm}\label{ccinf}  For a 4-manifold $M$ with $b^+=1,$ given a compatible pair  $(J,\w),$ one can inflate along a $J$-holomorphic curve $Z$, so that there exist a symplectic form $\w'$ compatible with  $J$  such that $[\w']= [\w]+ t PD(Z),t\in [0,\mu)$ whenever $\w'(Z)>0$  or, equivalently, $\mu= \infty$ if $Z\cdot Z\ge0$ and $\mu= \frac{\w(Z)}{(-Z\cdot Z)}$ if $Z\cdot Z<0$.
\end{thm}

  \subsection{ Keeping track of isotopy classes during the $J$-inflation.}\label{b+1inf}

Cohomologous forms are not known to be isotopic in the cases of blow-ups, in contrast to the minimal case. Recall that  $\mT_u$ is the space of symplectic forms in the cohomology class $u$ and  $\mT_{\w}$ is the space of symplectic forms isotopic to $\w$. When $u=[\w]$, $\mT_{\w}$ is a path connected component of  $\mT_u$.   \\

When performing $J$-compatible inflation, where $J$ is in the larger space $\mA_u$ (of almost complex structures compatible with some symplectic form in the cohomology class $u$), we must exhibit the following properties:

\begin{itemize}
  \item {\em All the path-connected components of both $\mA_{{u}}$ and  $\mT_{u}$ are homotopic to each other and there is a canonical bijection between the two sets of components}.

  By the  argument  in Lemma 4.1 in \cite{ALLP},  $\mA_{{\w}}$  is a path connected component of $\mA_u$ and is canonically homotopy equivalent to  $\mT_{\w}$.

In fact, $\mT_{\w}$ and $\mA_{ \w}$ correspond to each other under the
   canonical bijection between the sets of path connected components of $\mT_{u}$ and $\mA_{u}$ in Lemma 4.1 in \cite{ALLP}.

  \item {\em Performing inflations from one cohomology class to another preserves the path connected components of the spaces $\mA_{{u}}$ and  $\mT_{u}$.  }

  Let $J\in \mA_{ \w}$ such that $[\w]=u$.  Suppose that we perform the $b^+=1$ $J$-compatible inflation in Section \ref{tcinf} to $\omega$ to first get a symplectic form $\w'$ taming $J$.  Then by the cone results in \cite{LZ11},  one can conclude that $J$ is compatible with some symplectic form $\w''$ in the same cohomology class  $u'$ of $\w'$.
\end{itemize}

\begin{Claim}{\ }

{\renewcommand{\labelenumi}{{\em\roman{enumi})}}
  \begin{enumerate}
    \item For another $\tilde{J} \in \mA_{ \w}$, after performing the two steps in the $b^+=1$ $\tilde{J}$-compatible inflation we obtain a symplectic form
    $\tilde{\w''}$ in the same path connected component of $\mT_{u'}$ as $\w''$. \\
    \item Performing the two step $b^+=1$ inflation in the opposite direction, for any other $J'' \in \mA_{ \w''} $, produces an $\omega''' \in \mT_{ \w}$.
  \end{enumerate} }
\end{Claim}

We conclude that the inflation process does not change path connected components.

\begin{proof}

 First, let us choose the path connected component explicitly.
 The proof for the
 Claim i) and ii) for any other components follow from the first bullet in this section.

Let $J_{{\rm split}}$ be the  product  complex structure on $\Sigma_g\times \CC P^1$, pick a point $p_0\times 0$ on it  and consider its standard blow up inducing a holomorphic Lefschetz fibration   $\pi _{hol}: M_g \# \overline{ \CC P^2}\to (\Sigma_g,p_0)$.
Let $J_{{\rm std}}$ be the {\em holomorphic blow-up of $J_{{\rm split}}$ at $p_0\times 0$.}

As we will see below, for any cohomology class $u$, there is a preferred path connected component  ( denoted  $\mA_{ \w}$) of $\mA_u$ that contains $J_{{\rm std}}$. We choose all other $\mA_{ \w'}$ or $\mT_{ \w'}$ accordingly.

Namely,  for any $u$, there exists a symplectic form $ \bar{\w }$ with cohomology class $u$ so that  $J_{{\rm std}}$ is compatible with some form in the isotopy class, namely, $J \in \mA_{\bar{\w}}.$  The reason is we can always start from a K\"ahler ruled surface and inflate along the embedded fiber class curve to achieve any cohomology class in figure \ref{cone}. Hence the standard $J_{{\rm std}}$ tames some form in every class. Then, the comparison of the tame and compatible cones by Li-Zhang shows that $J_{{\rm std}}$ is compatible with some form in every class.
 Notice that this gives the canonical choice of the path connected component of $\mA_u$ for any $u$.  This proves claim i).

 To prove claim ii),   we use the canonical choice of path components given by $J_{{\rm std}}$:  Recall that  both $J_{{\rm std}}$ and $J''$ belong to the same path connected component $\mA_{ \w''}$. Define the product space $ P_{\w''} = \{(\w, J)\in \mT_u \times \mA_{ \w''} \quad | \quad  \w $  is compatible with $ J, [\w]=u\}.$  Consider the projection from $P_{\w''}$ to $\mA_{\w''}$, which amounts to taking a path connected component of $\mT_u$.

Notice that by section 3.5  \cite{MS17}, both the projection onto the $J$ and $\w$ factors have convex and hence contractible fibers. The space $\mA_{ \w''}$ is path connected.  Then the product space $P_{w''}$ is also path connected. Hence we know that the form   $\w'''$ described in the statement of claim ii) lives in the same path-connected component  as the original $\w.$ \\

\end{proof}

Now, that we can identify  a canonical path connected component containing $J_{{\rm std}}$ inside the space  $\mA_{u}$ and its corresponding path component inside  $\mT_{u}$ {\it we will abuse notation from here on and refer to  $\mA_{u}$  and  $\mT_{u}$  as being these path connected spaces.}

Note that such components are preserved in the following diagram that will be used later on during the of  proof of  Theorem \ref{highgenus}

  \begin{equation}
\begin{CD}
G_u @>>> \Diff_0(M)@>>> \mA_{ u}\\
@VVV @VVV @VVV \\
G_{u'} @>>>  \Diff_{0}(M)   @>>>  \mA_{ u'}.\\
\end{CD}
\end{equation}

\subsection{A historical detour to the minimal cases}\label{mininf}
It is informative to explain the proof strategy from McDuff \cite{McDacs} and Buse \cite{Buse11} in the minimal cases, to show how we can evolve our strategy in the non-minimal cases.

When the symplectic manifold is a minimal ruled surface the normalized symplectic cone is given by a line. The finite codimension strata are given by curve classes $B-kF$; these are labeled $\mA_{g,\mu} ^k $ and $ \mA_{g,\mu} ^{{\rm open}} $ respectively. McDuff showed that for each $J$ {\em there is a foliation of $M_g$ whose leaves are embedded $J$-holomorphic curves $F$}. She used this to show that $\mA_{g, \mu} \subset \mA_{g,\mu + \eps}$ for all $\mu, \eps > 0$, $g >0$, or for all $ \mu>1$, $\eps > 0$, $g =0$.

In her work,  this {\em right inclusion} is done regardless of strata; however, left inclusions ought to be proved stratum by stratum (including the open stratum) using inflation methods along embedded curves with positive self-intersections.

More specifically, she used curves in the section class $B$ (with a particular restriction on $\mu> \lfloor \frac{g}{2} \rfloor$ for the higher genus cases) to prove the left inclusion  $\mA_{g,\mu} ^{{\rm open}}\supset \mA_{g,\mu + \eps}^{{\rm open}}$.

The existence of sufficiently many embedded positive curves in the strata $ \mA_{g,\mu} ^k $ proved difficult in McDuff's work in the cases $g>0$, although possible in the rational cases. Later, Buse developed the inflation technique along negative curves in order to establish the left inclusions for strata with finite codimension. She completed the inclusions $\mA_{g,\mu} ^{k}\supset \mA_{g,\mu + \eps}^{k}$ for all $k>1$ and all $g>0.$

As a byproduct of these techniques and the homotopy fibration \eqref{fibacs}, Conjecture \ref{stabconj} holds for the following cases.

\begin{itemize}
    \item (McDuff\cite{McDacs}) When $g=0$, the homotopy type of the groups $G^0_{\mu}$ is unchanged on all intervals $(n, n+1], n$ for any nonzero natural number $n$.

    \item( Buse \cite{Buse11}) When $g>0$,  the homotopy type of the groups $G^g_{\mu}$ is unchanged on all intervals $(n, n+1]$ for any nonzero natural number $   n  \geq \lfloor g/2 \rfloor$.
\end{itemize}

Moreover, these stability results allowed McDuff to show that a homotopy colimit $G_{\infty, g}$ exists and to show that this group is homotopy equivalent to a smooth model group, namely the group $\mathcal{D}^0_g$ of smooth fiberwise diffeomorphisms of $M_g$. Among the topological consequences of this result we note that  $\mathcal{D}^0_g$ and $G_{\infty, g}$ (and consequently $G_{g,\mu}$ when  $   \mu  \geq \lfloor g/2 \rfloor$) {\em are connected.}

\section{ Stability of strata of $\mA_{u}$ in the one-point blow-up cases}\label{stabproof}

 We use a similar strategy for the non-minimal case. However, both the analysis of the curves and the inflation methods are more intricate here.

 In Section \ref{curves} we explain all the results yielding embedded curves in the one point blow up while in Section \ref{inflation} we detail all the inflation moves in the range $\mu >g$, both across the symplectic cone and within chambers needed to establish a stability result for the strata of almost complex structures.

Let us point out that unless necessary we will abuse the notation by omitting the genus subscript in our spaces and the blow-up number $n$ (which always equals 1 here).For example we will write $\mA_{u}$ instead of $\mA^n_{g, u}.$


We conclude with Section that provides the proof of the main stability Theorem \ref{highgenus}.

\subsection {Curves}\label{curves}
 The crux of the inflation arguments rely on the existence of embedded $J$ holomorphic curves in sufficient many homology classes, for any $J$. As the present paper only treats the one-point blow-up cases for any $g>0$; we lay out here the necessary existence results largely based on \cite{Zhang16}.



\begin{lma}\label{lma-ruled2}

 Let $(M,\omega)$ be a symplectic irrational ruled 4-manifold diffeomorphic to $M_g\# \overline{\CC P}^2$, and let $J$ be an $\omega$-compatible almost-complex structure. Then $M$ admits a {\bf Lefschetz fibration} given by a proper projection $\pi : M \to Y $ where Y is a smooth compact surface of  genus $g$ such that

  i) there is a singular value $y^{*} \in Y$ such that $\pi$ is a fibration over the space $Y  - y^{*}$,  with the fiber $\pi^{-1}(y)$, $y \in Y  - y^{*}$, represented by an embedded $J$-holomorphic  rational  curve in the class $F$;

  ii) the fiber $\pi^{-1}(y^{*})$ consists of the two exceptional $J$-holomorphic smooth rational curves in the classes $F-E$ and $E$.
\end{lma}

The main result in this subsection is Proposition \ref{inflation}.  To prove it, we  first need to show the existence of embedded $J$-holomorphic curves.

 \begin{prp}\label{curvesexist} Compendium of $J$-holomorphic curves on  $M=M_g \# \overline{\CC P}^2$.
\begin{enumerate}

\item For any $J$  $\in \mA_u$,  $M$ admits a Lefschetz fibration structure,  therefore, there are embedded $J$-holomorphic curves in the classes $F$, $F-E$ and $E$ by \cite{Zhang16}.

\item  For any $J$ in a positive codimensional stratum $\mA_{u, C}$, where $C$ is either $B-kF$ or $B-kF-E$, the class $C$ is guaranteed to have an embedded $J$-holomorphic curve representative by the very definition of the stratum.

\item  For any $J$ in $\mA_{u, {\rm open}},$  there is an embedded $J$-holomorphic curve in some class $B+xF$ or $B+xF-E$ for $0\le x\le g$.

\end{enumerate}
\end{prp}

\begin{proof}
\begin{enumerate}
    \item    Theorem 1.6 in \cite{Zhang16} gives the Lefschetz fibration where the smooth fibers are embedded $J$ holomorphic curves in the class $F$, and the singular fiber consists of the transversal union of an  the embedded $J$ holomorphic sphere in the class $F-E$ with an embedded $J$ holomorphic sphere  in the class $E$ .
    \item It follows by definition.
    \item It was proved by Li-Liu~\cite{LL01}  that, if $M = \Sigma_g\times
S^2$ or its blow-up, where $g > 0$ and $C = pB + qF$, and
$k(C) =  \frac 12 (c_1(C) + C^2) = \frac 12 (-K\cdot C  + C^2),$
where $k(C)$ means the virtual dimension of the moduli space of curves in class $C$.

Then
$$
\Gr(C) = (p+1)^g, \quad \mbox{provided that }\; k(C)
\ge 0.
$$

In particular, $\Gr(C) \ne 0$ provided that $q \ge g-1$.  When $g = 0$,
$\Gr(C) = 1$ for all classes $C$ with $p,q \ge 0$ and $ p+q > 0$.

 In general, for the genus $g$ cases we always have  $Gr(B+ gF)=2^g>0$, and expect to have $ k(B+gF)=\frac12([2B-2F-\cdots]\cdot [B+gF] +[B+gF]^2)= 2g+2-2g\ge 0.$

 We have a stable curve in the class $B+gF$ for any $J$.
  Since $J\in \mA_{u,{\rm open}}$, there are no curves with  self-intersection less than -1, meaning that for all curves, both $B$ and $F$ have non-negative coefficients.

 Now, looking at the stable representative, we know there is exactly one component where $B$ has coefficient 1, since the $B$-coefficient is non-negative. That component has to be in a class of the type $B+xF$, $B+xF-yE$ or $B+xF+yE$ for some $x\le g, y>0$, because the $F$-coefficient is non-negative.
 We will now show that such a component must be embedded and in the class   $B+xF$ or $B+xF-E.$
First we can eliminate the possibilities of $B+xF+yE, y>0$ because such curves pair negatively with $E$. Then we can also eliminate  the classes $B+xF-yE, y>1$ since they pair negatively with curves in class $F-E.$   Notice that  we do have $B+gF= B+xF +(g-x)F$ and
 $B+gF= B+xF-E +(g-x)F +E$ for $x\le g.$

 Note that such a component must be embedded. Indeed, if we assume there is a singular point, then we have an embedded curve in class $F$ or $F-E$ passing through the singular point.  The intersection number must be greater than 1, and this contradicts the fact that thehomological pairing with $F$ is $1.$

\end{enumerate}

\end{proof}

\subsection{Inflation}\label{s:inflation}

This section is concerned with proving the required stability (invariance) of strata in the spaces of almost complex structures $\mA_{u}$ when we consider the one-point blow-ups for all $g>0$.
 We will use the following {\bf inflation moves} to prove the stability of the strata of almost complex structures:

\begin{prp}\label{moves}

\begin{enumerate}[label=(\Alph*)]
    \item  {\bf Right horizontal inflation along $F$-class curves}:
    $\mA_{u}\subset \mA_{u'}$, whenever $u=[\mu,1,c]$, and $u'=[\mu+s,1,c]$ for all $\mu >g$, $s > 0$. Moreover, these inclusions preserve all strata including the open stratum.

 \item  {\bf Slant left inflation on open stratum along $B+xF$-class curves}\footnote{ The inflation along $B+xF$ could be achieved using curves $B+xF-E$ as as in Lemma \ref{formal} \label{note:FormalDir}}:
 For any $u'=[\mu',1,c'] $ with $\mu'>g$  and for any $g<\mu <\mu '$ there exists a $ 0<c<1$ so that
    $\mA_{u, open}\supset \mA_{u', open}$, where $u=[\mu,1,c]$, $u'=[\mu',1,c']$.

    \item  {\bf Slant left inflation on strata $ \mA_{u, B-kF} $ along $B-kF$-class curves }:
    For any
    $u'=[\mu',1,c'] $, $\mu'>k>g$ in the positive cone of $B-kF$,  and for any $g<k<\mu <\mu '$  in the positive cone of $B-kF$ there exists a  $ 0<c<1$ so that  $u=[\mu,1,c]$, $u'=[\mu',1,c']$ for all $\mu >g$,
    $\mA_{u, B-kF}\supset \mA_{u', B-kF}$. This means that we can inflate arbitrarily close to the left wall of the $B-kF$ positive cone.

     \item  {\bf Slant left inflation on strata $ \mA_{u, B-kF-E} $ along $B-kF-E$-class curves}:
    For any
    $u'=[\mu',1,c'] $  in the positive range of $B-kF-E$,  and for any $\mu$ with  $g<\mu <\mu' $  anywhere in the positive cone of $B-kF-E$ {\bf except in a shadow  open region in its left most chamber } there is a $0<c<1$ so that
    $\mA_{u, B-kF-E}\supset \mA_{u', B-kF-E}$.

    \item  {\bf Down vertical inflation  along $E$-class curves}:
    $\mA_{u'}\subset \mA_{u}$, whenever $u=[\mu,1,c]$, $u'=[\mu,1, c']$ for all $\mu >1$, $0< c <c'$. Moreover, these inclusions preserve the strata including the open stratum.

 \item  {\bf Two step up vertical inflation in the open stratum along   $B+xF$ , $F-E$ -class curves}: (see footnote~\ref{note:FormalDir} ) 
    $\mA_{u, open}\subset \mA_{u', open}$, whenever $u=[\mu,1,c]$, $u'=[\mu,1,c']$ for all $\mu \geq g+1$, $0< c<c'<1$. Note that this leaves out the leftmost chamber of the $B-gF$ positive cone.

    \item {\bf Two step up vertical inflation  in strata $\mA_{u, A-kF}$ along  $B-kF$, $F-E$ -class curves}:
    $\mA_{u, A-kF}\subset \mA_{u', A-kF}$, whenever $u=[\mu,1,c]$, $u'=[\mu,1,c']$ for all $\mu >g,\mu \geq k+1 $, $0< c<c'<1$.
    Note that this leaves out the leftmost chamber of the $B-kF$ positive cone.

     \item {\bf Two step up vertical  inflation in strata $\mA_{u, A-kF-E}$ along $F-E$,  $B-kF-E$ -class curves}:
    $\mA_{u, A-kF-E}\subset \mA_{u', A-kF-E}$, whenever $u=[\mu,1,c]$, $u'=[\mu,1,c']$ for all $\mu >g$, $0< c<c'<1$ in the positive cone of $A-kF-E.$

 \item {\bf Zig-Zag upwards inflation in the leftmost chamber of the $B-kF$ positive  repeatedly along $B-kF$ and $F-E$} :
    $\mA_{u, A-kF}\subset \mA_{u', A-kF}$, whenever $u=[\mu,1,c]$, $u'=[\mu,1,c']$ for all $\mu >g, k\leq \mu < k+1 $, $0< c<c'<1$.
    Note that this occurs in the leftmost chamber of the $B-kF$ positive cone.

  \item {\bf Zig-Zag upwards inflation in the leftmost chamber of the open stratum  repeatedly along  $B+xF$ (see footnote \ref{note:FormalDir} ) and $F-E$ }:
    $\mA_{u,open}\subset \mA_{u', open}$, whenever $u=[\mu,1,c]$, $u'=[\mu,1,c']$ for all $\mu >g, g\leq \mu < g+1$, $0< c<c'<1$.
    Note that this takes place in the leftmost chamber of the stability cone, close to the $B-gF$ line.

 \item {\bf Zig-Zag leftwards inflation in the left most chamber of the $B-kF-E$ repeatedly along the curves $E$ and $B-kF-E$}:

 For any 
    $u'=[\mu',1,c'] $  in the left most chamber of the positive cone of $B-kF-E$, i.e,  with $k<\mu' \leq k+1, \mu' -k-c'>0$  and for any $\mu$ with  $k<\mu <\mu' $  there is a  $0<c<1$ so that $ \mu -k-c>0$ and
    $\mA_{u, B-kF-E}\supset \mA_{u', B-kF-E}$.

    Note that this occurs in the leftmost chamber of the $B-kF-E$ positive cone.

\end{enumerate}
 \end{prp}
\begin{proof}

\begin{enumerate}[label=(\Alph*)]

\item By Lemma \ref{lma-ruled2}, we known that for each $J\in
\mA_{u, \mC}$ (including the open stratum) through each point of $M$ there is a stable $J$-holomorphic sphere representing the fiber class $F = [{\rm pt} \times S^2]$.
Inflating along it, we obtain a form in $s P.D[F] +[\mu, 1, c] $= $[\mu+s, 1, c]$, for all $s \in [0,\infty).$

 \item  By Proposition \ref{curvesexist}, we have an embedded curve in the class $B+xF$ for some $x\le g$. We begin with $u'=[\mu', 1, c'], \mu>g.$
After inflation we get a symplectic form in the cohomology class  $s P.D[B+xF] +[\mu', 1, c'] $= $s[x, 1,0]+[\mu', 1, c']$,
  that normalizes to
  \[ [\omega_s]=\biggl[\frac{sx+\mu'}{1+s}, 1, \frac{c'}{1+s}\biggr], \text{ for all } s \in [0,\infty).\]

  Note that $\displaystyle{\lim_{s\to \infty }\frac{sx+\mu'}{1+s} }=x\leq g,$ which means  that we are able to arrive at some point with $ u=[\mu, 1, c],$ for any $\mu>g.$

  \item
  By assumption, we have an embedded $J$- holomorphic curve in the class $B-kF.$ As above, we begin with $u'=[\mu', 1, c']$ and after inflation we obtain a form in the cohomology class $s P.D[B-kF] +[\mu', 1, c'] = s[-k, 1,0]+[\mu', 1, c]$,
  that normalizes to
  \begin{equation*}
     [\omega_s]=\biggl[\frac{-ks+\mu'}{1+s}, 1, \frac{c'}{1+s}\biggr], \text{ for all } s \in \biggl[0,\alpha= \frac{\mu '-k}{2k}\biggr)
  \end{equation*}\label{alphabound}

  Note that $\displaystyle{\lim_{s\to \alpha  }\frac{-ks+\mu'}{1+s} }=  k$ so we have covered all cases $\mu>k$.

 \item
 By assumption, we have an embedded $J$- holomorphic curve in the class $B-kF-E.$ Commencing with $u'=[\mu', 1, c']$ after inflation we obtain a form  with the cohomology class given by $ s P.D[B-kF-E] +[\mu', 1, c'] = s[-k, 1,1]+[\mu', 1, c]$,
  that normalizes to
  \begin{equation*}
    [\omega_s]=\biggl[\frac{-ks+\mu'}{1+s}, 1, \frac{c'+s}{1+s}\biggr], \text{ for all } s \in \biggl[ 0, \beta= \frac{ \mu'-k- c'}{2k+1}\bigg )
  \end{equation*}\label{betabound}

  Note that $\displaystyle{\lim_{s\to \beta  }\frac{-ks+\mu'}{1+s} }=  k + \frac{c'+\beta}{1+\beta} <k+1$ so we have covered all cases {\it except a shadow region in the left most chamber of the positive cone of $B-kF-E$ } which is completed in the point (K) below.

\item This is straightforward due to the presence of the embedded exceptional curve $E$ in each stratum.

\item
    Inflate along $F-E$ then follow with inflation along $B+xF$.    We get
    \[ [\w_{s, t}] =[\mu,1, c] +s [x,1,0] +t[1,0,1]=  [\mu+sx+t, 1+s, c+t],\]

Normalizing, we get forms in the class \[[\w_{s, t}]=\bigg [\frac{\mu+sx+t}{1+s}, 1, \frac{c+t}{1+s}\bigg ] \]which we want equal to $u'=[\mu,1,c']$ for all  $0< c<c'<1$.
This gives  a system in $s,t$ with $\mu+sx+t=\mu( 1+s)$ and
    $c+t= (c')(1+s).$

  Solving this system of two linear equations, we obtain $t=(\mu-x)s,$ and
  $s=\frac{c'-c}{\mu-x-c'}.$

    Note that the assumptions $  0\le x \le g, \mu\ge g+1$ ensures that  $s>0.$
    Moreover, $t$ automatically satisfies the upper bound in Theorem \ref{ccinf}, since $\omega_{s.t}(F-E)>0.$

  \item This move is similar to the one in the previous point. Inflate along $B-kF$ and $F-E$ to get
    \[[\w_{s, t}]=[\mu,1, c] +s [-k,1,0] +t[1,0,1]=[\mu-ks+t, 1+s, c+t],\]
    Proceeding as above, we get $s=\frac{c'-c}{\mu+k-c'}>0,  t =  (\mu + k)s>0.$

    Since $\mu \ge k +1$ (which means we leave of left-most chamber of the $A-kF$ positive cone),  any such $s$ satisfies the bound from Theorem \ref{ccinf}  $  \frac{\mu-k}{2k} >\frac{1}{ 2k} > s=\frac{c'-c}{\mu+k-c'}.$

    \item  Inflate along $B-kF-E$ and $F-E$, and require
      \[[\w_{s, t}]=[\mu,1, c] +s [-k,1,1] +t[1,0,1]=[\mu-ks+t, 1+s, c+s+t],\]
    where $t=(\mu + k)s$ and
    $c+t+s= c'(1+s).$

    Solving this system of two linear equations,
    we obtain $s=\frac{c'-c}{\mu+k+1-c'}>0,  t =  (\mu + k)s>0.$

    Note that for any choice of $0<c<c'<1$ the required $s$ is within the bounds $\omega_{s,t}(B-kF-E)>0$ of the Theorem \ref{ccinf}: Since  $\mu- k > c',$    we have    $  \frac{\mu-k -c}{2k+1} >\frac{c'-c}{ 2k+1} > s=\frac{c'-c}{\mu+k+1-c'}.$
        Hence we can always inflate along $B-kF-E$ for $s$ amount. We can  inflate along $F-E$ for $t$ amount, because $\omega_{s,t}(F-E)>0.$

  \item  When $k\le \mu < k +1,$ which means $u$ in   the left-most chamber, we need the following zigzag:

 \begin{figure}[ht]

\[
\xy
(0, 0)*{};(20, 0)* {}**\dir{--};
(0, 0)*{}; (0, -20)*{}**\dir{--};
(0, -20)*{};(20, 0)** {}**\dir{-};
(3, -15)*{};(17, -1)** {}**[PineGreen]\dir{-};
(12, -3)*{};(17, -1)** {}**[red]\dir{-};
(12, -3)*{};(14, -1)** {}**[PineGreen]\dir{-};
(9, -3)*{};(14, -1)** {}**[red]\dir{-};
(9, -3)*{};(11, -1)** {}**[PineGreen]\dir{-};
(6, -3)*{};(11, -1)** {}**[red]\dir{-};
(6, -3)*{};(8, -1)** {}**[PineGreen]\dir{-};
(3, -3)*{};(8, -1)** {}**[red]\dir{-};
 (-4, -12)*{{\text{B-pF}}};
  (10, -14)*{\rotatebox{45}{\text{B-pF-E}}};
 \endxy
\]

 \caption{zig-zag1}
  \label{zigzag}
\end{figure}

  Note that to reach $u'$ it suffices to reach any point above it and inflate along $E$, by the same moves of point (G).  Now let $\delta $ be $(1-c')/2.$ Then we can always do the following:

        	\begin{itemize}
		\item Step 0: inflate along $F-E$ for $ 1- \frac {\delta}{2} -c$,  to $[\tilde \mu, 1, \tilde c ],$ where $\tilde \mu = \mu +1- \frac {\delta}{2} -c$ and $\tilde c = 1-\frac {\delta}{2} >c'$ \footnote{ A convenient way of avoiding the $\epsilon, \delta$ in each step of the J-tame inflation process is to use the formal inflation as defined in \cite{Zha17}. }.
		
		\item Step 1:  inflate along $B-kF$ for $\epsilon>0$ small and we obtain
		$[\tilde \mu-k\epsilon, 1+\epsilon, \tilde c],$
		\item Step 1': inflate along $F-E$, for $\epsilon>0$ small and we obtain
		$[\tilde \mu-(k-1)\epsilon, 1+\epsilon, \tilde c+\epsilon].$ Notice that $$\frac{\tilde c+ \epsilon}{1+\epsilon} >\frac{\tilde c}{1} > c'.$$

	\item  Repeat the above two steps for $i$ times, one gets a symplectic form $\w_i$ compatible with $J$ such that $\w_i(B)/ \w_i (F) < \tilde \mu- i(k-1)\epsilon,$ and $\w_i(E)/ \w_i (F) > \tilde c>c'.$

			\item   $\cdots$

			\item   Eventually, we can decrease $\tilde \mu$ to the desired $\mu$ and obtain a form $[\mu, 1, \tilde c],$ where $\tilde c>c'$.

\end{itemize}
\item This zig-zag follows the same steps as above.

\item
Pick a $J\in \mA_{u',B-kF-E}$ and  $[\w']=u=[\mu', 1, c']$ in the left most chamber of the positive cone of $A-kF-E$, i.e. $k+c'<\mu' < k+1$.

Our aim is to show that for any $k< \mu< \mu'$ there is a small $c$  and $J$ taming $\w$ with $[\w]=[\mu, 1,  c]$ where $k+c< \mu'<\mu.$

  \begin{figure}
 \[
\xy
  (30, -14)*{\rotatebox{45}{\text{B-qF-E}}};
   (48, -12)*{{\text{B-(q+1)F}}};
(20, -20)*{};(40, 0)** {}**\dir{--};
(40, 0)*{}; (40, -20)*{}**\dir{-};
(20, -20)*{};(40, -20)* {}**\dir{--};
(37, -19)*{};(32, -18)** {}**[VioletRed]\dir{-};
(37, -19)*{};(37, -12)** {}**[blue]\dir{-};
(32, -19)*{};(32, -18)** {}**[blue]\dir{-};
(32, -19)*{};(27, -18)** {}**[VioletRed]\dir{-};
(22, -18)*{};(27, -19)** {}**[VioletRed]\dir{-};
(27, -19)*{};(27, -18)** {}**[blue]\dir{-};
\endxy
\]
  \caption{zig-zag2}
  \label{zigzag2}
\end{figure}

To do so, we do a sequence of moves down  as in point (E) that make the size of the exceptional curve very small alternated with moves slant leftwards as in point (D).
We just have to show that the areas of the forms in this sequence can be made as close to $k$ as we want. We show that each two steps, a combination of a move $(E)$ followed by a move  $(D)$ yield a symplectic form in a class $u_i=[\mu_i,1, \epsilon_i]$ in the  positive cone of $B-kE-E$, where $u_0=[\mu',1, c']$. Note that each first step inflation along $E$ decreases the sizes of the exceptional classes so we can assume that the sequence $\epsilon_i $ is monotonously decreasing to $0$. Moreover, following the bound $ \beta$ from step $D$ we can pick a sequence $\delta _i$ monotonously decreasing to $0$ so that the sizes $\mu_i$ satisfy the recursive relation $ \mu_{i+1}=  k + \frac{\epsilon_i+\beta_i}{1+\beta_i}-\delta_i $; where  $\beta_i= \frac{ \mu_i-k- \epsilon_i}{2k+1}$.

One can easily verify that the sequence $\mu_i$ is monotone and larger than $k$ so after going to a limit in the recursive relation we obtain that it approaches $k$. That means that we can reach our target $\mu >k$ after a finite sufficiently large number of steps.

\end{enumerate}

 \end{proof}

\begin{lma}\label{formal} For any $ J$  in $\mA_{u, open} $  with $u=[\mu,1,c], \mu >g$  that admits a symplectic embedded curve in a class $B+xF-E $ for some $x\le g$, one can reach any class $[\omega_\lambda]$ in the the formal direction of a $B-xF$ by inflating first along the curve $B+xF-E$ following with inflation along $E$.

\end{lma}

\begin{proof}

The proof is straightforward. Indeed we need to reach a class $[\omega_\lambda]=\biggl[\frac{\lambda x+\mu}{1+ \lambda}, 1, \frac{c}{1+\lambda }\biggr] $ as in point (B) in Proposition \ref{moves}. To do that we inflate for time $\lambda$ along the curve $B-xF-E$ to obtain  $ [\omega'_\lambda]=\biggl[\frac{\lambda x+\mu}{1+ \lambda}, 1, \frac{c+\lambda}{1+\lambda }\biggr] $ then follow by inflation along $E$ to reduce the area of the exceptional class to the desired size.

\end{proof}

\begin{prp}\label{inflation}

\begin{enumerate}

\item  $\mA_{u_1,\mC}= \mA_{u_2,\mC},$  for any $u_1=[\mu_1, 1, c_1]$, $u_2=[\mu_2, 1, c_2]$ in the positive cone of $\mC$, $\mu_i>g$ for all $\mC \subset S^{<0}$. %

\item For the open stratum, $\mA_{u_1, {\rm open}}= \mA_{u_2, {\rm open}},$ whenever $u_1=[\mu_1, 1, c_1]$, $u_2=[\mu_2, 1, c_2]$, and $\mu_i>g$.

\end{enumerate}

 In brief, any two forms in the same chamber have the same strata at every level. The strata of $\mA_{u}$ and $\mA_{u'}$   that are labeled by the same curve are the same, whenever $\mu_i>g.$

\end{prp}

\begin{proof}
The proof is immediate if one follows the inflation moves from                                   Proposition \ref{moves}.
In more detail: the vertical inflation is   (E) for moving downward. (F)   (combined with (J) for $\mu$ in the left most chambers) deals with moving up for open strata. For the rest of vertical moving up:  (G)   (combined with (I) for small  $\mu$ range) deals with $B-kF$ strata and  (H) deals with $B-kF-E$ strata.   Part (A) deals with all the rightward horizontal moves. For the leftward horizontal moves,  the open strata is handled in (B) of Proposition 4.3.   (D) of Proposition 4.3 deals with the cases when $B-kF-E$ is represented and (C) (combined with (K) for small range of $\mu$) of   deals with the cases when $B-kF$ is represented.

The following table summarizes what we need to follow to establish all needed double inclusions. The general idea is that moving right (toward infinity) is always easy, but moving left is harder. We start with a point on the left, name a target in a chamber, and perform large-scale or small-scale moves to hit the target.\\


\begin{table}[ht]
\begin{center}
\resizebox{\textwidth}{!}{\renewcommand{\arraystretch}{2}\begin{tabular}{||c c c c||}
\hline\hline
Direction  &  Strata         &
Class to inflate &
Size/Note \\[0.5ex]
\hline\hline
  $\downarrow$ &                                                       $\mA_{u, \mC} $ or  $\mA_{u, {\rm open}}  $   &
$E$                                              &
\parbox{1.7in}{ For all $c\in (0,1)$, $\mu >0.$} \\[0.5ex]
\hline
    $\uparrow$                                                    &  $\mA_{u, {\rm open}}$   &
$B+xF,$  $x\le g,$ and $F-E$                                                   &
\parbox{1.7in}{ Two-step inflation for $\mu>  g +1$, and use zigzag for   $g <\mu \le  g +1.$ } \\[0.5ex]
\hline
$\uparrow$                                                       &  $\mA_{B-kF-E}$ &
$B-kF-E$                                                 &
\parbox{1.7in}{Two-step inflation for $\mu> k +c$.} \\[0.5ex]
\hline
$\uparrow$                                                      &  $\mA_{B-kF}$   &
\parbox{1.05in}{ $B-kF, F-E$}                 &
\parbox{1.7in}{ Two-step inflation for $\mu> k +1$, and use zigzag for   $k <\mu \le  k +1$.} \\[0.5ex]
\hline
$\longrightarrow$                                                       &  Any strata     &
$F$                                                            &
\parbox{1.7in}{For  $\mu >0.$} \\[0.5ex]
\hline
$\longleftarrow$                                                    &  $\mA_{u, {\rm open}}$   &
$B+xF,$  $x\le g$                                                   &
\parbox{1.7in}{ One-step inflation for $\mu> g$.  } \\[0.5ex]
\hline
$\longleftarrow$                                                       &  $\mA_{B-kF-E}$ &
$B-kF-E$   and $E$                                              &
\parbox{1.7in}{One-step inflation for $\mu> k +1$, and use zigzag for   $k +c <\mu \le  k +1$.} \\[0.5ex]
\hline
$\longleftarrow$                                                      &  $\mA_{B-kF}$   &
\parbox{1.05in}{ $B-kF$}                 &
\parbox{1.7in}{ One-step inflation for $\mu> k$.} \\[0.5ex]
\hline\hline
\end{tabular} }
\caption{Inflation process}\label{inftable}
\end{center}
\end{table}


 \end{proof}

\subsection{Proof of Theorem \ref{highgenus}}

Proposition \ref{stabofsymp} following Corollary 2.3 of \cite{McDacs} yields the stability of the symplectomorphism group:

\begin{prp}\label{stabofsymp}   For any $u=[\mu,1,c_1], u'= [\mu+\eps,1,c_2],$ and $u''=[\mu+\eps+\eps',1,c_2], \mu>g, \eps,\eps' > 0$ there are  maps
$\mA_{u} \to \mA_{u'}$ and $G_{u} \to G_{u'}$, well defined up to homotopy, that make the following diagrams commute up to homotopy:
\[
\begin{array}{ccccccc}
{(a)} & & G_{u} &\to & \Diff_0(M_g \# \overline{ \CC P^2}) & \to & \mA_{u}\\
& & \downarrow & & \;\downarrow = & & \downarrow\\
& & G_{u'} &\to & \Diff_0(M_g \# \overline{ \CC P^2}) & \to & \mA_{u'},\\
& & & & & & \\
{ (b)} & &
G_{u} &\to & G_{u'} & & \\
& & & \searrow &\downarrow & & \\
& & && G_{u''}& &
\end{array}
\]
\end{prp}
\begin{proof}
  The maps $\mA_u \to \mA_{u'}$ are
the inclusions $\mA_u \subset \mA_{u'}$.
Since $G_u$ is the fiber of the map
$\Diff_0(M_g \# \overline{ \CC P^2}) \to \mA_{u}$, there are induced maps $G_u \to
G_{u'}$ making diagram $ (a)$ homotopy commute.  The rest is straightforward.

\end{proof}

The main Theorem \ref{highgenus} follows  from  Proposition \ref{moves} along with the above Proposition \ref{stabofsymp}. The first statement of \ref{highgenus} is a direct consequence of Proposition \ref{moves} because $\mA_u$ stay the same when $u$ moves in the stable chambers. For the  second statement of  Theorem \ref{highgenus}   By Proposition \ref{inflation}, the complement of $\mA_u$ in $\mA_{u'}$ is a union of suborbifolds of codimension higher than $cod_u= 2\lfloor \mu \rfloor -2g +2$.  Then by transversality, $\pi_i(\mA_u)$ and $\pi_i(\mA_{u'})$ are the same for any $i \le   2\lfloor \mu \rfloor -2g +1$.  Then by the long exact sequence of the diagram in Proposition \ref{stabofsymp}, for all $i \le   2\lfloor \mu \rfloor -2g$ we have $\pi_i(G_{u,g}^1)= \pi_i(G_{u',g}^1)$.    with $[\w]=u, [\w']=u'.$

\section{Lefschetz fibrations and topological colimit}\label{s:out}

The stability Theorem \ref{highgenus} grants us that the topological colimit $G^{1}_{\infty,g}$ exists for each horizontal line fixing the blow-up size.

 We use the relationship between the space of Lefschetz fibrations and the space of almost complex structures to establish a smooth diffeomorphism model for  $G^1_{\infty,g}$. Then we show that this smooth diffeomorphism model is disconnected and hence concludes that  $G^1_{\infty,g}$   is disconnected.

Proposition \ref{stabofsymp} shows that the homotopy colimit exists.

\subsection{Symplectic isotopy problem for one point blowup}
Now we prepare and give the proof of  Theorem \ref{tlimitintro}.
Note that $\mA_u$ is an open subset of $\mA_{u'}$, for any where $u'=[\mu',1, c],$ s.t.  $\mu'=\mu + \eps$ for all $\eps > 0$. The system of inclusions $i_{\epsilon}: \mA_u \hookrightarrow \mA_{\mu'} $ fits into the framework (described for example in Hatcher Ch 4.G \cite{Hatcherbook}) of a
 homotopy colimit construction and furthermore $\lim_{\mu\to \infty} \mA_u$ of the spaces $\mA_u$ is
homotopy equivalent to the union $\mA_\infty = \cup_\mu \mA_u$.

Recall $J_{{\rm split}}$ is the  product  complex structure on $\Sigma_g\times S^2.$
  As mentioned before,  $J_{{\rm std}}$ is  {\bf the holomorphic blow-up of $J_{{\rm split}}$.} 

\begin{lma}

There is a map $\Diff_0(M_g \# \overline{ \CC P^2}) \to \mA_\infty$  which induces a homotopy fibration, with $G^1_{g,\infty}$  as homotopy fiber.

\end{lma}

\begin{proof}
Because $J_{{\rm std}}$ is compatible with some $\om_{\mu}\in \mT_u$ the map $\Diff_0(M_g \# \overline{ \CC P^2}) \to
\mT_\mu$ lifts to
$$
\Diff_0(M_g \# \overline{ \CC P^2}) \to(\mT_u,\mA_u):\quad \phi\mapsto (\phi_*(\om_\mu),
\phi_*(J_{{\rm std}})).
$$
Composing with the projection to $\mA_u$ we get a map
\[
\Diff_0(M_g \# \overline{ \CC P^2}) \longrightarrow \mA_u\]
\[\phi\mapsto \phi_*(J_{{\rm std}})
\]
that is not a fibration but has homotopy fiber $G_u$.

Then by Proposition \ref{stabofsymp} (b), we are able to construct an action $\Diff_0(M_g \# \overline{ \CC P^2}) \to \mA_\infty$ which is compatible with all actions  $\Diff_0(M_g \# \overline{ \CC P^2}) \to\mA_u$ as described above.

\end{proof}

To understand $\mA_\infty$, let us first introduce the space
$\Fol$ of  Lefschetz fibrations of $\Si_g \times S^2\# \overline{ \CC P^2}$ as in Definition \ref{singfol}. In particular, this space only contains  fibrations with one nodal fiber defined as follows:

\begin{dfn}\label{singfol} The space $\Fol$ consist of all the  {\bf genus zero Lefschetz fibrations }  $ \pi: M_g \# \overline{ \CC P^2} \rightarrow  \Sigma_g$   with generic smooth embedded spherical fibers in the $F=[{\rm pt} \times S^2]$ class, and exactly one singular fiber  with two embedded exceptional spherical components each in the classes $E$ and $F-E$. Moreover, we require that the complement of the singular fiber is a smooth fiber bundle over $ \Sigma_g \setminus pt$. Note that these are non-relatively minimal genus zero Lefschetz fibrations, with only one separating nullhomologous vanishing cycle (see \cite[section 8]{GS99}),  which are just one-point blow ups of smooth trivial $S^2-$ fibrations  over a genus $g$ surface.
\end{dfn}

\begin{rmk} Zhang's results from Lemma \ref{lma-ruled2} imply that for each $J$ there is such genus zero  Lefschetz fibration with $J$-holomorphic fibers.
\end{rmk}

We will refer as $\Ff_{\rm std}$ to be the standard blow-up Lefschetz fibration by $J_{{\rm std}}$-holomorphic fibers. Note that by a blow-down of the complex structure, we obtain the split complex structure on $\Sigma_g \times S^2$, and the induced fibration is the split fibration by the spheres.

\begin{lma}
 Let $\Fol_0$ be the path connected component of $\Fol$ that contains $\mF_{{\rm std}}$.  $\mA_\infty$ is weakly homotopic to $\Fol_0$.
\end{lma}

\begin{proof}

Observe that there is a map $\mA_\infty \to \Fol_0$ given by taking
$J$ to the Lefschetz fibration of $M_g \# \overline{ \CC P^2}$ by $J$-spheres in class $F$ or $F-E$.  Notice that the map is surjective.  Standard arguments
in \cite{MS17} Ch 2.5 show that this map is a
fibration with contractible fibers.   Hence it is a homotopy equivalence.\\
\end{proof}

\begin{lma}\label{tranfol}
There is a transitive action of $\Diff_0(M_g \# \overline{ \CC P^2})$ on $\Fol_0$.

\end{lma}

\begin{proof}

This follows from the fact that there is a transitive action of (orientation preserving)  $\Diff^+(M_g \# \overline{ \CC P^2})$ on $\Fol$.
As explained below, the existence of such transitive action closely mimics the proof of such transitive action action for the minimal case ruled surfaces (see Proposition 1.1 in McDuff \cite{McDacs}). Specifically, we claim that  the classification result (see Proposition 6.2.3 in \cite{MS17}) of isomorphic genus zero fibrations over a genus $g$ surface  yields a similar classification for non-relatively minimal  Lefschetz fibrations with a single blow up singularity over a genus g surface.

For more details, recall  that (cf.  \cite[Section 8]{GS99}), similar to smooth fibrations,  two Lefschetz fibrations $\pi: M \to (\Sigma_g,p)$ and $\pi' : M' \to (\Sigma_g', p')$ are isomorphic if there is an orientation preserving diffeomorphism  $H: M  \to M'$ and $h: ( \Sigma_g, p) \to (\Sigma_g', p')$ with  $h \circ \pi  =  \pi ' \circ H$ that takes the singular point to singular point.

Consider one such Lefschetz fibration  $\pi: M_g \# \overline{ \CC P^2}\to (\Sigma_g,p)$,  characterized by a single separating null-homologous vanishing cycle, as in Definition \ref{singfol}. Then we claim it is  isomorphic  to the standard holomorphic Lefschetz fibration $ \mF_{std}$  presented as $\pi _{hol}: M_g \# \overline{ \CC P^2}\to (\Sigma_g,p_0)$.      To see that, take $ \Sigma_g  = D_p \cup (\Sigma_g \setminus D_p)$, where we view both the disk $D_p$ and $\Sigma_g \setminus D_p$ as  surfaces with boundary, each yielding two local fibrations $ \pi_{|D_p}$  and $ \pi_{|\Sigma_g \setminus D_p}$ (one singular one smooth) over surfaces with boundaries that have the same clutching data on the boundaries as the corresponding  minimal smooth fibration obtained by blowing down over the disk. Each one of these two local Lefschetz fibrations are diffeomorphic (See  Section 3  of Fuller \cite{Fuller03}) to the corresponding local fibrations of the holomorphic  $\pi_{hol}$. Following the standard bundle theory (cf \cite[section 6.2]{MS17}), the clutching maps are unique up to homotopy. Hence the two local Lefschetz fibrations trivializations  can be glued using the clutching data from their corresponding blow downs to yield  a global diffeomorphism  transiting the given fibration to the standard one.

 A path variation of this argument easily shows that if  a Lefschetz fibration is in the same path component as the standard $\mF_{\rm std}$ then the diffeomorphisms constructed are in  $ \Diff_0(M_g \# \overline{ \CC P^2})$ ; thus we  have  a transitive action of $\Diff_0(M_g \# \overline{ \CC P^2})$  on $ \Fol_0  . $
\end{proof}

  Hence  there is a fibration sequence
\begin{equation}\label{folpre}
  \Dd\cap \Diff_0(M_g \# \overline{ \CC P^2}) \to
\Diff_0(M_g \# \overline{ \CC P^2}) \to \Fol_0,
\end{equation}

where $\Dd$ is the diffeomorphism preserving the fibers in the fibration $\mF_{{\rm std}}.$ We denote this fiber group by $\Dd^1_g$.

\begin{dfn}\label{fibergp}
$\Dd^1_g$ consists of the elements in the identity component of the diffeomorphisms which fit into the commutative diagram $$
\begin{array}{ccc}\label{diagfibergp}

M_g \# \overline{ \CC P^2}& \stackrel \phi{\to} &M_g\# \overline{ \CC P^2}\\
\downarrow & &\downarrow\\
  (M_g, p, F_p) & \stackrel {\phi'} {\to} & (M_g, p, F_p)\\
\downarrow & &\downarrow\\
  (\Sigma_g, {\rm p_0}) & \stackrel {\phi''}{\to}  & (\Sigma_g,{\rm p_0}).
\end{array}
$$

\end{dfn}

Here $p$ is the intersection point $E\cap (F-E)$ of the singular fiber. The first level of the downward arrow means that we contract the $E$ component. We abuse notation here to call $p_0$ the point in $M_g$ {\em after} contracting the curve $E$.

On the second level, $\phi'$ is a diffeomorphism of $M_g$ that keeps the point $p$ fixed and fixes the fiber $F_p$ passing through $p$ fixed as a set. Thus, it preserves all smooth fibers in the standard Lefschetz fibration.

The section $\Sigma_g$ is the holomorphic curve $B_{{\rm std}}$ with respect to the standard complex structure. The map downward is obtained by blowing down the exceptional sphere, and then projecting down to the section curve.

\begin{prp}\label{tlimit}
\begin{enumerate}
    \item $\Dd^1_g$ is weakly homotopic to  $G^1_{\infty,g}$.

\item The group $\Dd^1_g$ is disconnected when $g\ge 2$.

\item When $\mu\to \infty$, s.t. $\pi_i(G^1_{u,g})=   \pi_i(G^1 _{\infty,g})$ for $i \leq \min\{ \cod_u\}-1,$ and hence the groups $G^1_{u,g}$ are disconnected for $g \geq 2$.
\end{enumerate}

\end{prp}

\begin{proof}
For statement (1), note the equation \eqref{folpre} fits into the
commutative diagram:
$$
\begin{array}{ccc}
  \Diff_0(M_g \# \overline{ \CC P^2})  & \to  & \mA_\infty\\
\downarrow & & \downarrow\\
\Diff_0(M_g \# \overline{ \CC P^2}) & \to & \Fol_0,
\end{array}
$$
where the upper map is given as before by the action $\phi\mapsto \phi_*(J_{{\rm std}})$.  Hence there is an induced homotopy equivalence from the homotopy fiber  $G^1_{\infty,g}$  of the top row to the fiber
$\Dd^1_g$ of the second.

To prove statement (2), first note that we have the following Birman exact sequence (cf. \cite{Bir69}) when $(g,n), g>0, n>0$ is not $(1,1)$:

$$ 1 \longrightarrow \pi_1(\Sigma_{g,n-1}) \longrightarrow \Gamma(g,n) \longrightarrow \Gamma(g,n-1) \longrightarrow 1.$$

Here $\Gamma(g,n)$ means the mapping class group of $\Sigma_g$ fixing $n$ points.

In our case, we are looking at $n=1$:
\[ 1 \longrightarrow \pi_1(\Sigma_{g}) \longrightarrow \Gamma(g,{\rm pt}) \longrightarrow \Gamma(g) \longrightarrow 1.\]

  Note that we can write the following fibration
  \[  \Diff(\Sigma_g, p_0)   \longrightarrow \Diff(\Sigma_g) \longrightarrow \Sigma_g,\]

  so that Birman exact sequence (restricted to the identity part of $\Gamma(g)$) becomes the associated long exact sequence:
\[  \Diff(\Sigma_g, p_0) \cap \Diff_0( \Sigma_g) \longrightarrow \Diff_0(\Sigma_g) \longrightarrow \Sigma_g.\]
The above sequence gives an element in the identity component of $\Diff(\Sigma_g)$ but not in the identity component of $\Diff(\Sigma_g, p_0)$, where $p_0$ is the point we will blow-up.
It can be made explicit in the following way: choose a path $\alpha(t)\subset  \Diff(\Sigma_g), t \in [0, 2\pi]$, pushing $p_0$ a   homotopically non-trivial loop on $\Sigma_g.$
 Now $\alpha(0)={\rm id}$ and $\alpha(2\pi) \in \Diff(\Sigma_g, p_0) \cap \Diff_0( \Sigma_g) $  and note that $\alpha(2\pi) $ is  the desired element.

 Next, we lift the path $\alpha(t)$ into $M_g \# \overline{ \CC P^2}$. To do that, first fix $M_g$, $\Sigma_g$ and choose $J_{\rm split}$. There is a natural family $\alpha(t)\times {\rm id} \subset \Diff_0(M_g),$ which acts on the fibers in the trivial manner.  For each of $t$, we have a product complex structure on $M_g$ by pulling back  $J_{\rm split}$ by $\alpha(t)\times {\rm id}$.  We are going to obtain a family of complex structures by blowing up at the points $\alpha(t)|_{p_0} \in M_g$.   This gives us a loop of complex structures $J_t$ on $M_g \# \overline{ \CC P^2}$  where $J_0=J_{{\rm std}}$. Note that by \cite{Zhang16},  each $J_t$   gives rise to a Lefschetz fibration $\mF_t$, as in Definition \ref{singfol}. Geometrically, $\mF_t$ is a loop in $\Fol_0$ starting with the standard Lefschetz fibration $\mF_{\rm std},$ pushing the marked point $p$ along a homotopicaly non-trivial circle on the standard section $\Sigma_g$ for time $t\in[0,2\pi]$.

By the transitivity Lemma \ref{tranfol}, we can use a path $\phi_t$ in $\Diff_0(M_g \# \overline{ \CC P^2})$ to push $\mF_0,$ so that $\phi_t\circ \mF_0= \mF_t.$    Note that $\phi_t$ in $\Diff_0(M_g \# \overline{ \CC P^2})$ pushes $\Ff_{std}$ along this loop.

 Now we focus on the diffeomorphism $\phi_{2\pi}$.  First note that $\phi_{2\pi}$ preserves the Lefschetz fibration $\mF_{\rm std},$ since the fibration $\mF_{2\pi} =\mF_0=\mF_{\rm std}.$  Hence  $\phi_{2\pi} \in \mD^1_g$.   Also, the above paragraph gives an explicit isotopy of $\phi_{2\pi}$ to the identity map in  $\Diff_0(M_g \# \overline{ \CC P^2})$, through the path $\phi_t.$

 We now show that $\phi_{2\pi}$ is not isotopic to {\rm id} in $\mD^1_g$. Suppose there is an isotopy to {\rm id}, then by path lifting of the fibration \ref{folpre},  we would have a fiber-preserving element in $\Diff_0(M_g \# \overline{ \CC P^2})$, so that it is isotopic to identity through a path in $\mD^1_g$. Furthermore, this path pushes the given fibration along the lifting of the loop $\mF_t, t\in[0,2\pi]$.  Now apply the diagram of Definition \ref{fibergp}. We would have an isotopy that would in turn give an isotopy of $(\Sigma_g, p)$, connecting the time $2\pi$ diffeomorphism to identity.  This is a contradiction against the Birman exact sequence. Hence statement (2) holds.

Statement (3) follows from a similar argument as in the stability  Proposition \ref{stabofsymp}.

\begin{equation}\label{Gcomm}
\begin{array}{ccccc}
 G^1_{u,g} & \to & \Diff_0(M_g \# \overline{ \CC P^2}) & \to & \mA_{u}\\
 \downarrow & & \;\downarrow  & & \downarrow\\
 G^1_{u,\infty} &\to & \Diff_0(M_g \# \overline{ \CC P^2}) & \to & \mA_{\infty},\\
\end{array}
\end{equation}

Recall the Definition  \ref{wall} of ${\cod_u}= max \{\cod_A ,  A\in \mS^{< 0}_u, u\cdot A \geq 0 \}$.   Note that the complement of the image of the inclusion map $\mA_u \hookrightarrow \mA_\infty$ is the union of suborbifolds of codimension at least $cod_u+2$. Then,     transversality results in orbifolds imply that $\pi_j(\mA_u)  =\pi_j( \mA_\infty)$ for all $j \le cod_u.$  Then by the above commutative diagram \eqref{Gcomm}, the connecting map of homotopy long exact sequence induces an isomorphism between $\pi_i(G^1_{u,g})$ and $   \pi_i(G^1 _{\infty,g})$ for $i \leq \min\{ \cod_u\}-1.$  Hence the groups $G^1_{u,g}$ are disconnected for $g \geq 2$.

\end{proof}
Note that our  Theorem \ref{tlimitintro} and  Corollary \ref{discon} follow from  Proposition \ref{tlimit}
\subsection{Discussion and Remarks}

\begin{rmk}
  When $g=0$, one can blow-up $S^2\times S^2$ at $k$ points with equal sizes. It is shown in \cite{LLW15} that when $k\le 3$, $ G^k_{u, 0}$ is connected for all $\w$. When $k>3$, $ \pi_0 G^k_{u, 0}$ is  a braid group of $k$ strands on spheres (cf. \cite{LLW3}).  This follows the same pattern as  $\Diff(S^2,k)$, which is the diffeomorphism group of $S^2$ fixing $k$ points.  In particular, when we take the one-point blow-up of $S^2\times S^2$, all arguments in the current paper apply. Even though we have a model for the colimit group, we are not able to show it is connected, but the loop that appears in our proof of Proposition \ref{tlimit} (2)  must give rise to a trivial mapping class since  $\Diff(S^2,1)$ is connected.\\
\end{rmk}

\begin{rmk}

It remains an open question whether the colimit group for $g=1$ is connected or not.  Shevchishin-Smirnov showed in  \cite{SS17elliptic} that there exists a symplectomorphism called elliptic twist along a (-1) torus on the minimal ruled surface or its one-point blow-up.  They further showed that when $\mu\to \infty$ the elliptic twist is isotopic to identity, and a non-trivial mapping class appears in the first chamber in our Figure \ref{cone}.  The two potential non-trivial symplectic mapping classes, elliptic twists as in \cite{SS17elliptic} and the loop in Proposition \ref{tlimit} (2), fail to appear in the colimit group.  In  \cite{SS17elliptic} this is because the (-1) tori has a positive area.  In our case, it is because adding one puncture to a torus does not change its mapping class group.  Hence we still do not know the connectivity status of the group $D_1^1$, although we conjecture that it is connected.

\end{rmk}

\begin{rmk}\label{fiberdehn}

We speculate that the non-trivial mapping class constructed in Proposition \ref{tlimit} has a geometric construction from the fibered Dehn twist by Biran-Giroux \cite{BirGir} Here is some evidence: In our topological colimit,  first take a neighborhood $\mathcal{N}(\Sigma_g)$ of the smooth section $\Sigma_g$ in Definition \ref{fibergp}. With a small perturbation, we can assume that the boundary $ \partial \mathcal{N}(\Sigma_g)$ meets transversely with each section of the fibration $\mF_{{\rm std}}$ at one circle, making it a Seifert fibered 3-manifold. Heuristically, the lifting of the loop $\alpha(t)$ in Proposition \ref{tlimit} acts on $ \partial \mathcal{N}(\Sigma_g) \times [0,1]$,  which looks like a ``topological fibered Dehn twist'':
\begin{align*}
T: \partial \mathcal{N}(\Sigma_g) \times [0,1] &\to \partial \mathcal{N}(\Sigma_g) \times [0,1], \\
(x,t) \hspace{.52cm}&\longmapsto (x\cdot[g(t)\hspace{.12cm}\mbox{\rm mod}\hspace{.1cm} 2\pi],t),
\end{align*}

Here $g:[0,1] \rightarrow \mathbb{R}$ is a smooth function such that $g(t)=0$ near $t=1$ and $g(t)=2\pi$ near $t=0$.

It is natural to conjecture that this construction is a mapping class isotopic to our construction in Proposition \ref{tlimit} part (2).  Furthermore, it is an interesting question to probe whether this topological fibered Dehn twist at $\mu=\infty$ is a direct limit of the symplectic fibered Dehn twist of  Biran-Giroux \cite{BirGir} for $\mu$ finite.

\end{rmk}

\begin{rmk}\rm  Implicit in the above argument is the following
 description of the map  $G^1_{\infty,g} \to \Dd^1_g$.  Let $\Jj_\mu$ denote the
 space of all almost complex structures compatible with $\om_\mu$.
Since the image of the group $G_\mu$ under the map $\Diff_0(M_g \# \overline{ \CC P^2}) \to
\mA_\mu$ is contained in $\Jj_\mu$, there is
a commutative diagram
$$
\begin{array}{ccc} & & \Dd^1_g\\
&  & \downarrow\\
G_\mu &\stackrel{\io}\longrightarrow & \Diff_0(M_g \# \overline{ \CC P^2})\\
\downarrow & & \downarrow\\
\Jj_\mu & \longrightarrow & \Fol_0.
\end{array}
$$
Because $\Jj_\mu$ is contractible, the inclusion $\io:
G_\mu \to \Diff_0(M_g \# \overline{ \CC P^2})$  lifts to a map  $
\Tilde\io: G_\mu \to \Dd^1_g.$  Now take the limit to get $G_\infty \to \Dd^1_g$.
\end{rmk}

\nocite{}
\printbibliography

\end{document}